\documentclass[a4paper,11pt]{scrartcl}
\usepackage[utf8]{inputenc}
\usepackage{graphicx,tikz}
\usepackage{amsmath,amssymb,amsthm}

\newtheorem{de}{Definition}[section]
\newtheorem{theo}{Theorem}[section]
\newtheorem{prop}[theo]{Proposition}
\newtheorem{cor}[theo]{Corollary}

\newtheorem{conj}{Conjecture}

\title{On Multi-label radio Labelings of Graphs}
\author{Olivier Togni\\ LIB UR 7534, Université Bourgogne Europe, Dijon, France}

\begin{document}

\maketitle

\begin{abstract}
We propose and study a multi-labeling generalization of the $k$-radio labeling problem of graphs: given a graph $G$ and integers $k$ and $m$, assign to each vertex of $G$ a set of $m$ integer labels such that distinct labels assigned to a vertex differ by at least $k+1$, and the labels of any two distinct vertices $x$ and $y$ pairwise differ by at least $k+1-d(x,y)$, with $d(x,y)$ being the distance between vertices $x$ and $y$. We prove general bounds on the minimum span of such a labeling and more refined results for the case $k=1$. For paths and cycles, we derive sharp bounds on the minimum span when $k$ is small. We finally propose a conjecture for the minimum span on the path and for an almost-tight bound on the cycle.


\end{abstract}

\section{Introduction}
Frequency assignment problems in wireless communication systems have long been modeled using graph labelings. In this setting, vertices represent transmitters, and labels represent frequencies or communication channels, with a required separation between labels depending on the distance between the corresponding transmitters. This naturally leads to distance-constrained labelings, such as \(L(d_1,\ldots,d_t)\)-labelings~\cite{BPT00}, in which labels assigned to vertices at distance \(i\) must satisfy a prescribed separation \(d_i\).
In particular, radio labelings 
provide a natural example, with a separation requirement that decreases linearly with graph distance: a $k$-radio labeling is an $L(k,k-1,\ldots,1)$-labeling.

To better model modern wireless communication systems that use multiple channels or frequencies, as well as multi-stream technologies such as MIMO, we introduce a multi-label extension of radio labeling. Each vertex is assigned \(m\) labels, which can be interpreted as the frequencies or channels assigned to a transmitter. We impose separation constraints not only between labels assigned to distinct vertices, but also between labels assigned to the same vertex. 
\begin{de}
 For integers $m,k\geq 1$, an {\em $(m,k)$-radio labeling} of a graph $G$ is an assignment $\ell: V(G) \rightarrow \mathbb{N}^m$ such that for any two (possibly equal) vertices $x$ and $y$, and any $c_x\in\ell(x)$, $c_y\in\ell(y)$ with $(x,c_x)\neq (y,c_y)$, we have $$|c_x-c_y|\geq k+1-d_G(x,y),$$ where $d_G(x,y)$ is the distance between $x$ and $y$ in $G$. 
 
The minimum span (difference between the largest and the smallest label used) among all $(m,k)$-radio labelings is called the {\em $(m,k)$-radio number} and is denoted by $\lambda_k^m(G)$.
\end{de}


Thus, labels assigned to the same vertex must be separated by at least \(k+1\), while labels assigned to vertices at distances \(1,2,\ldots,k\) must be separated by at least \(k,k-1,\ldots,1\), respectively.

As an illustration, consider the graph of Figure~\ref{f1} with a $(3,2)$-radio labeling of span 18. Observe that the labels assigned to a vertex need not form a consecutive block in the global ordering of labels, i.e., labels assigned to different vertices may be (partially) interleaved.


\begin{figure}[ht]
\centering
\begin{tikzpicture}[
    vertex/.style={
        draw,
        rounded corners,
        minimum width=20mm,
        minimum height=8mm
    }
]

\node[vertex] (v0) at (0,0) {$\{10,14,18\}$};
\node[vertex] (v1) at (4,0) {$\{0,3,6\}$};
\node[vertex] (v2) at (8,0) {$\{8,12,16\}$};
\node[vertex] (v3) at (6,-2) {$\{1,4,7\}$};

\draw (v0)--(v1)--(v2)--(v3)--(v1);

\end{tikzpicture}
\caption{\label{f1}A graph with $(3,2)$-radio labeling of span 18.}
\end{figure}

Our model combines two graph-theoretic approaches to frequency assignment. For \(m=1\), it reduces to the usual \(k\)-radio labeling introduced by Chartrand et al. under the name radio coloring; see their survey~\cite{CZ07}. For $k=2$, it coincides with the well-studied $L(2,1)$-labeling problem introduced in 1992 by Griggs and Yeh~\cite{GY92}. 
Most of the existing results on $k$-radio labelings concern either values of $k$ equal to or close to the diameter of the graph~\cite{Liu08,JL12,BVZ17}, or arbitrary values of $k$ for simple graph classes such as paths, cycles, trees~\cite{KKT07}, product graphs~\cite{KKT08}, and distance graphs~\cite{CEH+13}. The case $k=3$ was studied by Chia et al.~\cite{CKL+11} and more recently by several groups of authors~\cite{Cal23,Zha24,BC26,HK26}.

On the other hand, assigning several labels to each vertex is the principle underlying graph multicoloring and, in particular, the weighted multicoloring problem with a co-site constraint introduced by Gerke \cite{Ger00}, where several frequencies may be assigned to a transmitter and separation constraints are imposed both between labels assigned to the same vertex and between labels assigned to adjacent vertices. Our model extends this framework by imposing distance-dependent constraints beyond adjacent vertices, following the philosophy of distance-constrained radio labelings, while assigning a constant number of labels per vertex.

We propose in this paper general bounds for the $(m,k)$-radio number of arbitrary graphs, depending in particular on the clique number and on the maximum degree. We compute the $(m,1)$-radio number of any bipartite graph and of any cycle. We determine the $(m,2)$-radio number of the path asymptotically and give an upper bound for the $(m,2)$-radio number of the cycle and for the $(m,3)$-radio number of the path. Finally, we propose two conjectures that unify and extend the results obtained for paths and cycles. 
The paper is organized as follows: Section 2 presents the notation and some general bounds on the $(m,k)$-radio number of a graph. Section 3 focuses on the case $k=1$, providing bounds and exact results for bipartite graphs and cycles. In Section 4, we prove results on the $(m,2)$- and $(m,3)$-radio numbers of the path and of the cycle. Section 5 concludes the paper by stating a conjecture on the $(m,k)$-radio number of the path and of the cycle.

\section{General bounds}
In the remainder of the paper, we use the terms "label" and "color" interchangeably.
In order to formalize the distance between integers in sets, we introduce the following notation: Two integers $p$ and $q$ are said to be {\em $s$-separated} if $|p-q|\ge s$. By extension, a set $X$ is {\em $s$-separated} if the elements of $X$ are pairwise $s$-separated and is {\em $s$-separated} from a set $Y$ if any element of $X$ is $s$-separated from any element of $Y$.

With this notation, an $(m,k)$-labeling $\ell$ of a graph $G$ satisfies that $\ell(x)$ is $(k+1-d_G(x,y))$-separated from $\ell(y)$ for any two vertices $x$ and $y$.


We can relate the $(m,k)$-radio number and the classical $k$-radio number in the following way:
\begin{prop}\label{propGen}
 For any graph $G$ and any integers $k,m$ and $p$, $1\le p < m$, we have $$\lambda_k^m(G)\le \lambda_k^{m-p}(G)+\lambda_k^{p}(G)+k+1.$$
\end{prop}

\begin{proof}
We construct a labeling $f$ of $G$ of span $\lambda_k^{m-p}(G)+\lambda_k^{p}(G)+k$ from an $(m-p,k)$-radio labeling $\ell_1$ of $G$ of span $\lambda_k^{m-p}(G)$ with labels from $\{0,1,\ldots, \lambda_k^{m-p}(G)\}$ and from a $(p,k)$-radio labeling $\ell_2$ of $G$ of span $\lambda_k^{p}(G)$ with labels from $\{\lambda_k^{m-p}(G)+k+1,\ldots, \lambda_k^{m-p}(G)+\lambda_k^{p}(G)+k+1\}$ by setting for any $v\in V(G)$,
$$f(v)=\ell_1(v)\cup \ell_2(v).$$
It is easily verified that $f$ is an $(m,k)$-radio labeling, since each of $\ell_1$ and $\ell_2$ satisfies the required constraints and $\ell_1, \ell_2$ are $(k+1)$-separated.
\end{proof}

\begin{cor}
 For any graph $G$ and any integers $k,m\ge 1$, $$\lambda_k^m(G)\le m\lambda_k(G) + (m-1)(k+1).$$
\end{cor}

In order to propose a bound for the $(m,k)$-radio number of a graph $G$ in terms of its clique number $\omega(G)$, we first determine the exact value of the $(m,k)$-radio number of the complete graph $K_n$ for any $k,m$ and $n$.
\begin{prop}
For any $n\ge 2$, $k\ge 1$ and $m\ge 1$,  $$\lambda_k^m(K_n) = k(mn-1).$$
\end{prop}

\begin{proof}
Let $f$ be an $(m,k)$-radio labeling of $K_n$ and let $f_1\le f_2\le ... \le f_{mn}$ be the ordered sequence of labels assigned to the vertices. First, as all vertices are neighbors of each other, we have $f_{i+1}\ge f_i +k$ for any $1\le i\le mn-1$. Hence, $\lambda_k^m(K_n)\ge f_{mn}-f_1 \ge k(mn-1)$. Next, let $V(K_n)=\{v_0,\ldots,v_{n-1}\}$. We define an $(m,k)$-radio labeling of $K_n$ of span $k(mn-1)$ by setting $f(v_i)=\{ik, (n+i)k,\ldots, ((m-1)n+i)k\}$ for $0\le i\le n-1$. It is easily seen that any two labels assigned to the same vertex are $nk$-separated (and therefore $(k+1)$-separated), while the sets assigned to distinct vertices are pairwise $k$-separated.
\end{proof}

\begin{cor}
    For any graph $G$ and integers $k\ge 1$ and $m\ge 1$, $$\lambda_k^m(G) \ge k(m\omega(G)-1).$$
\end{cor}

Now, to obtain a bound depending on the maximum degree, we first determine the $(m,k)$-radio number of a star.

\begin{prop}\label{pK1n}
For any $n\ge 3$, and $m,k\ge 2$,  $$\lambda_k^m(K_{1,{n-1}}) = (mn-1)(k-1)+2m-2,$$ except for $n=3$ and $k=2$ for which $\lambda_2^m(K_{1,{2}}) =5m-2$.
\end{prop}

\begin{proof}
Let $s=(mn-1)(k-1)+2m-2$.
We first prove that $\lambda_2^m(K_{1,{2}}) \le s+1=5m-2$ and $\lambda_k^m(K_{1,{n-1}})\le s$ for $(n,k)\ne (3,2)$. Let $V(K_{1,n-1})=\{v_0,v_1,\ldots,v_{n-1}\}$, with $v_0$ being the center of the star. 

If $(n,k)= (3,2)$, then the labeling $\ell$ defined by 
$$\begin{aligned}
\ell(v_0)&=\{5i: 0\le i\le m-1\}, \\ 
\ell(v_1)&=\{5i+2: 0\le i\le m-1\}, \mbox{ and }\\
\ell(v_2)&=\{5i+3: 0\le i\le m-1\},
\end{aligned}$$ 
is an $(m,2)$-radio labeling of $K_{1,2}$ of span $5m-2$.

If $(n,k)\ne (3,2)$, then let $\ell$ be defined by, for $1\le i\le n-1$,

$$\begin{aligned}
\ell(v_0)&= \{0,s-(m-2)(k+1),s-(m-3)(k+1),\ldots, s-(k+1),s\},  \mbox{ and }\\
\ell(v_i)&=\{i(k-1)+1, (n-1+i)(k-1)+1,\ldots, ((m-1)(n-1)+i)(k-1)+1\}.
\end{aligned}$$ 

We now prove that $\ell$ satisfies the required separation constraints. First, in order to see that $\ell(v_0)$ is $(k+1)$-separated, we show that $s-(m-2)(k+1)\ge k+1$. Indeed, $s-(m-2)(k+1)= (mn-1)(k-1)+2m-2 -(m-2)(k+1)=m((k-1)(n-1)-2)+k+1\ge k+1$ since $k\ge 2$ and $n\ge 3$.
Moreover, for each $i$, $\ell(v_i)$ is $((k-1)(n-1))$-separated and $(k-1)(n-1)\ge k+1$ except for $(n,k)=(3,2)$.

Second, we check that labels $c_1$ and $c_2$ assigned to two distinct vertices $x$ and $y$ are sufficiently separated. 
First, observe that the labels of vertices $v_i$ are between the first and the second labels of $v_0$ in increasing order, i.e., between $0$ and $\ell_0=s-(m-2)(k+1)$. The smallest label of any vertex $v_i$ is the smallest label of $v_1$, that is, $k$, and is thus $k$-separated from $\ell(v_0)$. The largest label of any vertex $v_i$ is the largest label of $v_{n-1}$, i.e.,  $l_1=m(n-1)(k-1)+1$. We have 
$$\ell_0-\ell_1 = (mn-1)(k-1)+2m-2 - (m-2)(k+1) - m(n-1)(k-1)+1= k.$$

Hence $\ell(v_i)$ is $k$-separated from $\ell(v_0)$.
If now $x=v_i$ and $y=v_j$ for some $1\le i<j\le n-1$, then $c_1=(p_1(n-1)+i)(k-1)+1$ and $c_2=(p_2(n-1)+j)(k-1)+1$ for some $p_1,p_2$, $0\le p_1,p_2\le m-1$. Since $1\le j-i\le n-2$, then 
$$|c_2-c_1|= |((p_2-p_1)(n-1)+j-i)(k-1)|\ge k-1.$$ 
Thus $\ell(v_i)$ is $(k-1)$-separated from $\ell(v_j)$.

Now we prove that no better $(m,k)$-radio labeling $\ell$ exists. Let the $mn$ labels assigned by $\ell$ be ordered as $c_1\le c_2\le\cdots\le c_{mn}$. By definition, the minimum difference between two labels $c_i$ and $c_j$ must be $k-1$ if they are assigned to two different leaves,  $k$ if one is on $v_0$ and the other on a leaf, and $k+1$ if they are both assigned to the same vertex. Consequently, $c_{i+1}-c_i \ge k-1$ for any $i$. Let $\delta_i=c_{i+1}-c_i- (k-1)\ge 0$.

The span of $\ell$ is $$c_{mn}-c_1 = (mn-1)(k-1)+\sum_{i=1}^{mn-1} \delta_i.$$

We show that in any case, $\sum_{i=1}^{mn-1} \delta_i\ge 2m-2$, and that this minimum is attained only if both $c_1$ and $c_{mn}$ are assigned to $v_0$.
The $m$ labels of the center $v_0$ will give this $2m-2$ sum. 
Assume label $c_i$ is assigned to $v_0$. If $i>1$ then $\delta_{i-1}\ge 1$ if $c_{i-1}$ is assigned to a leaf, and $\delta_{i-1}\ge 2$ if $c_{i-1}$ is also assigned to $v_0$, in which case this $+2$ is counted as a contribution of $+1$ for both label $c_{i-1}$ and $c_i$. Moreover, if $i<mn$, then $\delta_{i}\ge 1$ if $c_{i+1}$ is assigned to a leaf and $\delta_{i}\ge 2$ if $c_{i+1}$ is also assigned to $v_0$ (counted also $+1$ for both labels). In summary, each label assigned to $v_0$ contributes a unit in some $\delta_i$ for each side where it possesses a neighboring label in the ordered sequence of labels (thus two units if the label is neither $c_1$ nor $c_{mn}$ and one unit otherwise). If two consecutive labels $c_i, c_{i+1}$ of the ordered sequence are assigned to the center $v_0$, then the contribution $\delta_i\ge 2$ allows us to count one unit for each of the two labels. As there are $m$ labels assigned to the center, and at most two of them can be at an end of the sequence, then the sum of contributions is at least $2(m-2)+2=2m-2$.


Now, we still have to treat the case $n,k= 3,2$, and to prove that in this case $\sum_{i=1}^{3m-1} \delta_i\ge 2m-1$. Suppose to the contrary, that there exists an $(m,2)$-labeling of $K_{1,2}$ with $\sum_{i=1}^{3m-1} \delta_i = 2m-2$. Then, by the above argument, labels $c_1$ and $c_{3m}$ must be assigned to the center $v_0$.

Suppose first that there exists some $i$, $2 < i < 3m-1$, for which at least three consecutive labels $c_{i-1},c_i,c_{i+1}$ in the ordered sequence are assigned to leaves. If $c_{i-1}$ and $c_i$ are assigned to the same leave, then by definition of the radio-labeling, $c_i-c_{i-1}\ge k+1=3$, hence $\delta_{i-1}\ge 2$. If $c_{i}$ and $c_{i+1}$ are assigned to the same leave, we similarly have $\delta_i\ge 2$. This contribution is independent of the $2m-2$ units of contributions of $ v_0$'s labels. Therefore, in total, we will have in this case $\sum_{i=1}^{3m-1} \delta_i\ge 2m$, a contradiction. Otherwise, labels $c_{i-1},c_i,c_{i+1}$ alternate between the two leaves, hence $c_{i-1}$ and $c_{i+1}$ are assigned to the same leave and thus $c_{i+1}-c_{i-1}\ge 3$, implying $\delta_{i-1}+\delta_{i}\ge 1$. We thus obtain $\sum_{i=1}^{3m-1} \delta_i\ge 2m-1$, a contradiction.

Now, if for any $i$, at least one of $c_{i-1},c_i,c_{i+1}$ is assigned to the center $v_0$, then we can see that there are not enough labels of $v_0$ to do that. Indeed, since the two extremal labels are assigned to the center $v_0$, there remain only $m-2$ labels of $v_0$ to delimit blocks of consecutive leaf labels. Hence, we can place at most $2(m-1)=2m-2$ leaf labels, but we have $2m$ such labels, a contradiction.

Therefore $\sum_{i=1}^{3m-1} \delta_i> 2m-2$ and thus, $\lambda_2^m(K_{1,{2}}) > 3m-1 + 2m-2= 5m-3$.

\end{proof}

Since any graph of maximum degree $\Delta$ contains a subgraph isomorphic to $K_{1,\Delta}$, Proposition~\ref{pK1n} gives the following bound for $\Delta\ge 2$:
\begin{cor}
    For any graph $G$ of maximum degree $\Delta\ge 2$ and for any integer $m,k\ge 2$, $$\lambda_k^m(G) \ge (m(\Delta+1)-1)(k-1)+2m-2.$$
\end{cor}

\section{$(m,1)$-radio labeling}
Quite naturally, the $(m,1)$-radio number of a graph is related to the $m$-fold chromatic number $\chi_m(G)$ and fractional chromatic numbers $\chi_f(G)$. The $m$-fold chromatic number $\chi_m(G)$ of a graph $G$ is the least $k$ such that vertices can be given $m$-subsets of $\{1,\ldots,k\}$ in such a way that adjacent vertices have disjoint sets of colors. The fractional chromatic number of $G$ can be defined by $\chi_f(G)=\inf_m \chi_m(G)/m$.

\begin{prop}\label{p1}
For any non-empty graph $G$ and integer $m\ge 1$, $$m\chi_f(G)-1\le \chi_{m}(G)-1 \le \lambda_1^m(G)\le m\chi(G)-1.$$
\end{prop}

\begin{proof}
For the lower bound, we use the fact that $\chi_m(G)\ge m\chi_f(G)$, see the above definitions. As an $(m,1)$-radio labeling is nothing but an $m$-fold coloring starting with color 0 instead of 1, we have $\lambda_1^m(G)\ge \chi_m(G)-1 \ge m\chi_f(G)-1$.

For the upper bound, let $c:V(G)\rightarrow \{0,\ldots, \chi(G)-1\}$ be a coloring of $G$. We define an $(m,1)$-radio labeling $\ell$ by setting 
$$\ell(v)=\{i\chi(G)+c(v) : 0\le i\le m-1\}.$$  
Then, trivially, the set assigned to any vertex is 2-separated (in fact $\chi(G)$-separated, and $\chi(G)\ge 2$ since $G$ is non-empty) and the sets of two adjacent vertices $x$ and $y$ are disjoint since $c$ is a proper coloring.

The span of this labeling is $(m-1)\chi(G) + \chi(G) -1 = m\chi(G)-1$. 
\end{proof}

Remark that one might be tempted to replace the bounds of Proposition~\ref{p1} by the equality $\lambda_1^m(G)=\lceil m\chi_f(G)\rceil-1$. However, this equality does not hold in general because of the additional 2-separation constraint on the colors assigned to each vertex. 
For instance for the Petersen graph $P$, we know that $\chi_f(P)=\chi_2(P)/2=5/2$. But all the 2-subsets of $\{1,\ldots,5\}$ are used on some vertex in any 2-fold 5-coloring of $P$, so no permutation of colors will 2-separate all subsets. However, there are graphs for which this equality holds. We will see next that this is the case for cycles.

For bipartite graphs, the bounds of Proposition~\ref{p1} simplify to an equality (since $\chi(G)=2$ implies $\chi_f(G)=2$):
\begin{cor}\label{c1}
For any bipartite graph $B$ and any $m\ge 1$, we have $\lambda_1^m(B)=2m-1$.
\end{cor}

For the cycle, we determine the exact value of $\lambda_1^m$:
\begin{prop}\label{l1mcycle}
For any $n\ge 2$ and $m\ge 1$,  $$\lambda_1^m(C_n) = \lceil m\chi_f(C_n)\rceil-1=\left\{
\begin{array}{ll}2m-1, & \mbox{ if $n$ is even},\\
\left\lceil\frac{2n}{n-1}m\right\rceil-1, & \mbox{ if $n$ is odd.}
\end{array}\right.$$
\end{prop}

\begin{proof}
Let $V(C_n)=\{v_0,\ldots,v_{n-1}\}$, with $v_i$ adjacent with $v_{i\pm 1 \bmod{n}}$. For even $n$, the result follows from Corollary~\ref{c1}.
For odd $n$, we have by Proposition~\ref{p1} $\lambda_1^m(C_n)\ge \lceil m\chi_f(C_n)\rceil-1= \lceil 2mn/(n-1)\rceil-1$.
To prove the equality for $n=2p+1$, we  exhibit an $(m,1)$-radio labeling of span $s=\lceil2mn/(n-1)\rceil-1= \lceil m(2p+1)/p\rceil-1=2m+\lceil m/p\rceil -1$ by considering separately the two cases $m\le p$ and $m>p$.

\paragraph{Case 1.  $m\le p$.} Note that $\lceil m/p\rceil=1$, thus $s = 2m$. Set the label of the vertex $v_0$ as
$$\ell(v_0)=\{1,3,\ldots, 2m-1 \}.$$

Now, for any $0\le i\le 2m-1$, the labels assigned to $v_{i+1}$ are obtained by a +1 shift (modulo $s+1=2m+1$) of the label of $v_i$, that is, we set,  
$$\ell(v_{i+1})=\{l+1 \bmod{2m+1}: l\in \ell(v_{i})\}.$$

For the remaining vertices, we alternate between the labels of $v_0$ and of $v_{2m}$: for any $1\le j\le p-m$,
$$\ell(v_{2m+2j-1})= \ell(v_0) \mbox{ and } \ell(v_{2m+2j})= \ell(v_{2m}).$$

Since $\ell(v_{2m})= \{0,2,\ldots, 2m-2\}=\ell(v_{2p})$, we have that $\ell(v_{2p})\ne \ell(v_0)$ and thus the labels of any two consecutive vertices $v_i$ and $v_{i+1}$ are disjoints. Hence $\ell$ is a $(m,1)$-radio labeling of $P_n$. An example for $m=3$ and $p=5$ is given in Figure~\ref{t1m1} (left).

\paragraph{Case 2. $m>p\ge 1$.} Let $r=\left\lceil\frac{m}{p}\right\rceil$ and define a labeling $\ell$ of $C_{2p+1}$ of span $s=2m+\lceil\frac{m}{p}\rceil-1$ by setting for any $i$, $0\le i\le 2p$:
$$
\ell(v_i)=L_i=\{a_{i,0},a_{i,1},\ldots,a_{i,m-1}\}
$$

where for $j\in\{0,\ldots,m-1\}$, $a_{i,j}=2j+\left\lfloor\frac{j+t_i}{p}\right\rfloor$, with $t_i=ip\mod{2p+1}$.

Now, we prove that $\ell$ is an $(m,1)$-radio labeling of $C_n$. We first observe that, for every $i$ and every $j<m-1$,

$$
a_{i,j+1}-a_{i,j}=2+\left\lfloor\frac{j+1+t_i}{p}\right\rfloor - \left\lfloor\frac{j+t_i}{p}\right\rfloor \in\{2,3\}.
$$

Consequently, each $L_i$ is $2$-separated.

Moreover, $a_{i,0}\geq0$ while

$$
\begin{aligned}
a_{i,m-1}
&\leq
2m-2+
\left\lfloor\frac{m-1+2p}{p}\right\rfloor\\
&=
2m+\left\lfloor\frac{m-1}{p}\right\rfloor\\
&=
2m+\left\lceil\frac{m}{p}\right\rceil-1
=s.
\end{aligned}
$$

Thus, every color belongs to $\{0,\ldots,s\}$.

It remains to be proved that two consecutive sets are disjoint. Since $t_{i+1}\equiv t_i+p\pmod{2p+1}$,  there are two cases.

Suppose first that $0\leq t_i\leq p$. Then $t_{i+1}=t_i+p$, and hence

$$
a_{i+1,j}=2j+\left\lfloor\frac{j+t_i+p}{p}\right\rfloor=a_{i,j}+1.
$$

Moreover, since consecutive elements of $L_i$ differ by at least $2$, we have $L_i\cap L_{i+1}=\emptyset$.

Now suppose that $p+1\leq t_i\leq2p$. In this case, $t_{i+1}=t_i-p-1$.
Thus, for every $j<m-1$, we have
$$
\begin{aligned}
a_{i+1,j+1}
&=
2j+2+
\left\lfloor
\frac{j+1+t_i-p-1}{p}
\right\rfloor\\
&=
2j+2+
\left\lfloor
\frac{j+t_i-p}{p}
\right\rfloor\\
&=
a_{i,j}+1.
\end{aligned}
$$

Hence $a_{i+1,j}<a_{i,j}<a_{i+1,j+1}$ and therefore $L_i\cap L_{i+1}=\emptyset$.

Finally, the same argument applies to the edge $v_{2p}v_0$. Indeed, $t_{2p}=2p^2\bmod{2p+1}=p+1$, so $t_0=t_{2p}-p-1$ and we obtain, as above, that $a_{0,j}< a_{2p,j} < a_{0,j+1}$ and $L_{2p}\cap L_{0}=\emptyset$.

Thus the sets $L_0,\ldots,L_{2p}$ define the required $(m,1)$-radio labeling of $C_{2p+1}$.
Figure~\ref{t1m1} (right) shows an example of this labeling on the cycle $C_7$ for $m=6$. 
\end{proof}

\begin{figure}
\begin{center}
\begin{tabular}{c|ccc}
$v_0$&1 & 3 & 5 \\ 
$v_1$&2 & 4 & 6 \\ 
$v_2$&3 & 5 & 0 \\ 
$v_3$&4 & 6 & 1 \\ 
$v_4$&5 & 0 & 2 \\ 
$v_5$&6 & 1 & 3 \\ 
$v_6$&0 & 2 & 4 \\ 
$v_7$&1 & 3 & 5 \\ 
$v_8$&0 & 2 & 4 \\ 
$v_9$&1 & 3 & 5 \\ 
$v_{10}$&0 & 2 & 4 \\ 
\end{tabular} 
\hspace*{2cm}
\begin{tabular}{c|cccccc}
$v_0$&0 & 2 & 4 & 7 & 9 & 11 \\ 
$v_1$&1 & 3 & 5 & 8 & 10 & 12 \\ 
$v_2$&2 & 4 & 6 & 9 & 11 & 13 \\ 
$v_3$&0 & 3 & 5 & 7 & 10 & 12 \\ 
$v_4$&1 & 4 & 6 & 8 & 11 & 13 \\ 
$v_5$&0 & 2 & 5 & 7 & 9 & 12 \\ 
$v_6$&1 & 3 & 6 & 8 & 10 & 13 \\ 
\end{tabular} 
\end{center}
\caption{\label{t1m1}An optimal $(3,1)$-labeling of $C_{11}$ of span 6 (on the left) and optimal $(6,1)$-radio labeling of $C_7$ of span $13$ (on the right).}
\end{figure}

\section{$(m,k\ge 2)$-labeling of paths and cycles}
The $(m,1)$-radio number of the path and of the cycle is determined by Corollary~\ref{c1} and Proposition~\ref{l1mcycle}, respectively.
In this section, we present lower and upper bounds for the $(m,2)$-radio number of the path that asymptotically coincide, and an upper bound for the $(m,2)$-radio number of the cycle. For $k=3$, we provide an upper bound for the $(m,3)$-radio number of the path.

\begin{prop}\label{pPn2}
For any $n\ge 1$ and $m\ge 1$,  $$\lambda_2^m(P_n) \le 5m-1.$$
\end{prop}
\begin{proof}

We partition the set $\{0,1\ldots, 5m-1\}$ into five subsets of size $m$ 

$$L_j=\{5i+j: i=0,\ldots, m-1\}, \qquad 0\le j\le 4.$$


Let $V(P_n)=\{v_0, v_1,\ldots, v_{n-1}\}$, with $v_i$ adjacent with $v_{i\pm 1}$. Define a labeling $f$ of $P_n$ by setting, for every $0\le i\le n-1$,
$$f(v_i)=L_{2i\bmod{5}}.$$
It is easy to see that $f$ is an $(m,2)$-radio labeling of $P_n$ of span $5m-1=\max(L_4)$. Indeed, if two labels $\alpha$ and $\beta$ are assigned to the same vertex, then by definition, they are $5$-separated; if they are given to two adjacent vertices, then $\alpha=5i+r$ for some $i$ and $0\le r\le 4$ and $\beta=5j+ (r+2) \bmod{5}$, therefore $(\beta -\alpha)\bmod{5} = r+2-r=2$; and if they are given to two vertices at distance 2, then $(\beta -\alpha)\bmod{5} = r+4-r=4$, hence $\alpha\ne\beta$.
\end{proof}

In order to see the sharpness of the above upper bound, we prove a lower bound for the path $P_n$ that coincides with the upper bound when $n$ is large enough.
\begin{prop}\label{pPinf}
For any $m\ge 1$, and sufficiently large $n$, $$\lambda_2^m(P_n) \ge 5 m -1.$$
\end{prop}

\begin{proof}
Assume, to the contrary, that there exists an $(m,2)$-radio labeling $\ell$ of $P_n$ of span $5m-2$. Partition the set of available colors into one block of size $4$ and $m-1$ blocks of size $5$:
\[
B_{0}=\{0,1,2,3\}
\]
and
\[
B_j=\{5j-1,5j,\ldots,5j+3\}, \qquad 1\le j\le m-1,
\]
For a block $B_j$, $1\le j\le m-1$, let
\[
n_t^j=\bigl|\{v_i:\ |\ell(v_i)\cap B_j|=t\}\bigr|,
\qquad t\in\{0,1,2\}.
\]

Since a vertex can receive at most two colors from $B_j$, we have
\[
n_0^j+n_1^j+n_2^j=n.
\]

Moreover, if a vertex receives two colors from $B_j$, then it must receive one of the pairs $\{5j-1,5j+2\}, \{5j,5j+3\}$, or $\{5j-1,5j+3\}$.

In each case, at least one neighboring vertex receives no color from $B_j$. If the pair $\{5j-1,5j+2\}$ or $\{5j,5j+3\}$ is given to some vertex, then both neighbors receive no color from $B_j$. Moreover, one can check that the pair $\{5j-1,5j+3\}$ cannot be given to two vertices at distance 1 or 2. Hence, one can associate injectively each vertex counted by $n_2^j$ with a distinct vertex counted by $n_0^j$, yielding
\[
n_2^j\le n_0^j.
\]

Therefore, the total number of colors from $B_j$ assigned to the vertices of $P_n$ is
\[
n_1^j+2n_2^j
\le n_1^j+n_2^j+n_0^j
=n.
\]
Consequently,
\[
\sum_{i=0}^{n-1} |\ell(v_i)\cap B_j|
\le n,
\qquad 1\le j\le m-1.
\]

Now consider the first block $B_0=\{0,1,2,3\}$.
Let
$a_i=|\ell(v_i)\cap B_{0}|$. We check all admissible configurations on seven consecutive vertices. First, such a configuration has at most three vertices having two colors from $B_{0}$. Table~\ref{tb2} shows all eleven configurations up to color isomorphism. $C_1$ is the only configuration with three vertices assigned two colors of $B_0$; $C_2-C_6$ are the five configurations with two vertices assigned two colors; $c_8-C_{10}$ are the only three configurations with exactly one vertex assigned two colors, and finally, $C_{11}$ is the only configuration with no vertex assigned two colors of $B_0$. We can observe from each configuration that every interval of seven consecutive vertices contains at most six colors from $B_{0}$, i.e., for any vertex $i$,  $0\le i\le n-7$,
\[
\sum_{t=i}^{i+6} a_t\le 6.
\]

It follows that
\[
\sum_{i=0}^{n-1} a_i=\sum_{i=0}^{n-1} |\ell(v_i)\cap B_{0}|
\le \frac67 n + O(1).
\]

Therefore, summing up the inequalities for $0\le j\le m-1$, we have 
\[
\sum_{j=0}^{m-1}\sum_{i=0}^{n-1} |\ell(v_i)\cap B_j| <\left(m-\frac17\right)n+O(1).
\]
For sufficiently large $n$, the right-hand side is strictly smaller than $mn$, a contradiction.
\end{proof}

\begin{table}[ht]
\begin{center}
\begin{tabular}{c|ccccccc}
$C_1$ & 0,3 & $\times$ & $\times$ & 0,3 & $\times$ & $\times$ & 0,3\\ 
$C_2$ & 0,3 & $\times$ & $\times$ & 0,3 & $\times$ & 1 & 3\\
$C_3$ & 0,3 & $\times$ & 1 & $\times$ & 0,3 & $\times$ & 1\\
$C_4$ & 0,3 & $\times$ & 1 & $\times$ & $\times$ & 0,3 & $\times$\\
$C_5$ & 0,3 & $\times$ & 1 & 3 & $\times$ & $\times$ & 0,3\\
$C_6$ & $\times$ & 0,3 & $\times$ & $\times$ & 0,3 & $\times$ & 1\\
$C_7$ & 0,3 & $\times$ & 1 & 3 & 0 & 2 & $\times$\\
$C_8$ & $\times$ & 0,3 & $\times$ & 1 & 3 & 0 & 2\\
$C_9$ & 2 & $\times$ & 0,3 & $\times$ & 1 & 3 & 0\\
$C_{10}$ & 0 & 2 & $\times$ & 0,3 & $\times$ & 1 & 3\\
$C_{11}$ & 1 & 3 & 0 & 2 & $\times$ & 3 & 0 
\end{tabular}
\end{center}
\caption{\label{tb2}The 11 configurations of colors from bloc $B_{0}$ on seven consecutive vertices (up to color isomorphism and where $\times$ means no color from $B_{0}$).}
\end{table}

We now turn our attention to the cycle, proving that for $k=2$ and under some conditions, the minimum span is the same or just one more than that of the path with the same order.
\begin{prop}
For any $m,p\ge 1$,  $$\lambda_2^m(C_{5p}) = 5m-1.$$
Moreover, for $2\le m\le 4$ and any $n\ge 28$, $\lambda_2^m(C_{n})\le 5m$.
\end{prop}
\begin{proof}
First, for $C_{5p}$, one can use the same labeling as the one of $P_n$ given in the proof of Proposition~\ref{pPn2}, which is still an $(m,2)$-radio labeling of $C_{5p}$ since $L_4$ is 2-separated from $L_0$ and labels of $L_4$ and $L_1$ are pairwise distinct.

Now, for $2\le m\le 4$, a labeling $f$ of $C_n$ of span $5m$ can be defined by combining two sequences of $m$-subsets of $\{0,\ldots, 5m\}$: one of size 5 and the other of size 8, as described in Table~\ref{tbm2}. It can be seen that for $m=2$ or $3$, the first two sets of the sequence of length 5 are equal to the first two sets of the sequence of length 8. Therefore, the two sequences can be used by $f$ one after another in any order along the cycle. For $m=4$, it is similar, but for the last two subsets of each sequence. Hence, these sequences can be used one after the other in any order along the cycle. Moreover, since 5 and 8 are co-prime, any $n\ge 28=4\times 7$ can be written as $n=5i+8j$ for some $i,j$, then one can always combine $i$ sequences of length 5 and $j$ sequences of length 8 to produce an $(m,2)$-labeling of $C_n$.

\begin{table}[ht]
\begin{center}
\begin{tabular}{cc}
0 & 5   \\
2 & 7   \\
4 & 9   \\
1 & 6   \\
3 & 8  \\
- & -\\
0 & 5  \\
2 & 7  \\
4 & 10 \\
0 & 6  \\
2 & 9  \\
4 & 7  \\
1 & 10 \\
3 & 8 
\end{tabular}\hspace*{1.5cm}
\begin{tabular}{ccc}
0 & 5 & 10 \\
2 & 7 & 12 \\
4 & 9 & 14 \\
1 & 6 & 11 \\
3 & 8 & 13 \\
- & - & - \\
0 & 5  & 10 \\
2 & 7  & 12 \\
4 & 9  & 14 \\
0 & 6  & 11 \\
2 & 8  & 13 \\
4 & 10 & 15 \\
1 & 6  & 12 \\
3 & 8  & 14
\end{tabular}\hspace*{1.5cm}
\begin{tabular}{cccc}
0 & 5 & 10 & 16    \\
2 & 7 & 13 & 18    \\
4 & 9 & 15 & 20    \\
1 & 6 & 12 & 17    \\
3 & 8 & 14 & 19    \\
- & - & - & - \\
0 & 5 & 10 & 16 \\
2 & 7 & 12 & 18 \\
4 & 9 & 14 & 20 \\
0 & 6 & 11 & 16 \\
2 & 8 & 13 & 18 \\
4 & 10 & 15 & 20 \\
1 & 6 & 12 & 17 \\
3 & 8 & 14 & 19
\end{tabular}
\end{center}
\caption{\label{tbm2}$(m,2)$-radio labelings of $C_5$ and $C_8$ for $m=2,3$ and $4$.}
\end{table}
\end{proof}


Following the same idea as in Proof of Proposition~\ref{pPinf}, we can bound the $(m,3)$-radio number of the path:
\begin{prop}\label{pPn3}
For any $n\ge 1$ and $m\ge 1$,  $$\lambda_3^m(P_n) \le 8m-1.$$
\end{prop}
\begin{proof}

We partition the set $\{0,1\ldots, 8m-1\}$ into height $m$-subsets 

$$L_j=\{8i+j: i=0,\ldots, m-1\}, \qquad 0\le j\le 7.$$


Let $V(P_n)=\{v_0, v_1,\ldots, v_{n-1}\}$, with $v_i$ adjacent with $v_{i\pm 1}$. Define a labeling $f$ of $P_n$ by setting, for every $0\le i\le n-1$, 
$$f(v_i)=L_{3i\bmod{8}}.$$  As in the proof of Proposition~\ref{pPn2}, one can check that $f$ is a $(m,k)$-radio labeling of $P_n$ of span $8m-1=\max(L_7)$.
\end{proof}

\section{Concluding remarks}
As shown by the results and proofs of Sections 2, 3 and 4, computing the $(m,k)$-radio number of a graph turns out to be very difficult, even for simple graphs such as paths. 

In order to obtain a upper bound for the $(m,k)$-radio number of the path, the method used both in Propositions~\ref {pPn2} and~\ref{pPn3} could possibly be applied for any $k$: Let $s=\left\lceil\frac{(k+1)^2}{2}\right\rceil$, and partition the set $\{0,\ldots,sm-1\}$ into $s$ $m$-subsets $L_j=\{si+j: i=0,\ldots, m-1\}$, $0\le j\le s-1$. Then assign these sets to the vertices of the path by setting $f(v_i)=L_{di\bmod{s}}$ for some well-chosen $d$.
For $k=4$, it seems that $d=5$ produces an $(m,4)$-radio labeling, and for $k=5$, $d=5$ or $7$ seems okay too. But, as $k$ increases, finding the good value of $d$ and checking that the defined labeling is an $(m,k)$-radio labeling is getting harder.

Note also that, similarly to the case $k=2$, it may be possible to prove a lower bound for the $(m,3)$-radio number of the path that coincides with the upper bound for sufficiently large $n$, i.e., that  $\lambda_3^m(P_n) \ge 8m-1$, but this would probably require checking more configurations on a larger sequence of consecutive vertices.

Motivated by our results for paths when \(k\in\{1,2,3\}\), together with computational evidence for larger values of \(k\), we propose the following conjecture. 
\begin{conj}
For any $k,m \ge 1$, there exists an integer $n_0$ such that for any $n\ge n_0$,
$$\lambda_k^m(P_n) = \left\lceil\frac{(k+1)^2}{2}\right\rceil m -1.$$
\end{conj}

Corollary~\ref{c1} and Propositions~\ref{pPinf} and \ref{pPn2} prove the conjecture for $k=1$ and $2$ and the bound of Proposition~\ref{pPn3} coincides with the conjectured value for $k=3$.
Moreover, note that this conjecture generalizes the one given in 2007 by Kchikech, Khennoufa, and Togni for $m=1$~\cite{KKT07}, which is still open.

For the cycle, our results and initial computations suggest that for sufficiently large order, an $(m,k)$-radio labeling can be found using at most one additional color than for the path of the same order. We state this as a conjecture:
\begin{conj}
For any $k,m \ge 1$, there exists an integer $n_0$ such that for any $n\ge n_0$,
$$\lambda_k^m(C_n) \le \left\lceil\frac{(k+1)^2}{2}\right\rceil m.$$
\end{conj}

\section*{Declaration of generative AI technologies use}
During the preparation of this work, ChatGPT was used to check for typos in the proofs and to improve the English formulation. After using this tool, the author reviewed and edited the content as needed and take full responsibility for the content of the published article.

\end{document}